\documentclass{imsart}

\RequirePackage[OT1]{fontenc}
\RequirePackage{amsthm,amsmath}
\RequirePackage[colorlinks,citecolor=blue,urlcolor=blue]{hyperref}

\usepackage[utf8]{inputenc}
\usepackage[T1]{fontenc}

\RequirePackage{amssymb,mathtools }
\RequirePackage{natbib}
\setcitestyle{authoryear,open={(},close={)}}
\startlocaldefs
\numberwithin{equation}{section}
\theoremstyle{plain}

\theoremstyle{remark}
\newtheorem{remark}{Remark}[section]

\newtheorem{theorem}{Theorem}[section]
\newtheorem{lemma}{Lemma}[section]
\newtheorem{proposition}{Proposition}
 
\endlocaldefs

\begin{document}

\begin{frontmatter}
\title{Three remarkable properties of the Normal distribution}
\runtitle{Three remarkable properties of the Normal}

\begin{aug}
\author{\fnms{Eric} \snm{Benhamou}\thanksref{a,b,e1,r1}}
\author{\fnms{Beatrice} \snm{Guez}\thanksref{a,e2,r1}}
\author{\fnms{Nicolas} \snm{Paris}\thanksref{a,e3,r1}}

\address[a]{A.I. SQUARE CONNECT, \\ 35 Boulevard d'Inkermann 92200 Neuilly sur Seine, France}
\address[b]{LAMSADE, Université Paris Dauphine, \\ Place du Maréchal de Lattre de Tassigny,75016 Paris, France}
\thankstext{e1}{eric.benhamou@aisquareconnect.com, eric.benhamou@dauphine.eu}
\thankstext{e2}{beatrice.guez@aisquareconnect.com}
\thankstext{e3}{nicolas.paris@aisquareconnect.com}
\thankstext{r1}{The authors would like to thank Laurent Tournier and 
Henrique de Oliveira for valuable comments on their paper.}

\runauthor{E. Benhamou, B. Guez, N. Paris}

\affiliation{A.I. SQUARE CONNECT and LAMSADE, Paris Dauphine University}
\end{aug}

\begin{abstract}
In this paper, we present three remarkable properties of the normal distribution: first that if two independent variables 's sum is normally distributed, then each random variable follows a normal distribution (which is referred to as the Levy Cramer theorem), second a variation of the Levy Cramer theorem (new to our knowledge) that states that two independent symmetric random variables with finite variance, independent sum and difference are necessarily normal, and third that normal distribution can be characterized by the fact that it is the only distribution for which sample mean and variance are independent, which is a central property for deriving the Student distribution and referred as the Geary theorem. The novelty of this paper is twofold. First we provide an extension of the Levy Cramer theorem. Second, for the two seminal theorem (the Levy Cramer and Geary theorem), we provide new, quicker or self contained proofs.
\end{abstract}



\end{frontmatter}

%%%%%%%%%%%
%%% INTRODUCTION
%%%%%%%%%%%

\section{Introduction}
The normal distribution is central to probability and statistics. It is obviously the asymptotic law for the sum of i.i.d. variables as stated by the Central Limit theorem and its different variants and extensions. It is also well know that the normal distribution is the only Levy $\alpha$-stable distribution with the largest value of the $\alpha$ coefficient, namely 2. This fundamental properties explains why the Wiener process plays such a central role in stochastic calculus. There are many more properties of the normal distribution. It is for instance the continuous distribution with given mean and variance that has the maximum entropy. 
In this paper, we present three remarkable properties of the normal distribution: first that if $X_1$ and $X_2$ are two independent variables with $X_1 + X_2$ normally distributed, then both $X_1$ and $X_2$ follow a normal distribution (which is referred to the Levy Cramer theorem), second a variation of the Levy Cramer theorem that states that if $X_1$ and $X_2$ are independent symmetric random variables with finite variance with $X_1+X_2$ and $X_1 - X_2$ independent, then both $X_1$ and $X_2$ are normal, and third that the normal distribution is characterized by the fact that it is the only distribution for which the sample mean and variance are independent, which is a central property for deriving the Student distribution and referred as the Geary theorem (see \cite{ Geary_1936}). The novelty of this paper is to provide new, quicker or self contained proofs of theses theorems as well as an extension of the Levy Cramer theorem. In particular, we revisit the proof first provided by \cite{Feller_1971} to give full details of it as the sketch of the proof provided was rather elusive on technical details in complex analysis. We also provide a quicker proof of the Geary theorem using the log caracteristic function instead of the caracteristic function originally used.
\\
The paper is organized as follows. We first present the Levy Cramer theorem. We provide a thorough proof of the elusive sketch provided by \cite{Feller_1971}. We then present another remarkable result that combines independence and normal law. We finally present a quick proof of the fact that the independence of the sample mean and variance is a characterization of the normal law.

\section{Levy Cramer theorem}
The Levy Cramer theorem states that if the sum of two independent non-constant random variables $X_1$ and $X_2$ is normally distributed, then each of the summands ($X_1$ or $X_2$) is normally distributed. This result stated \cite{Levy_1935} and proved by \cite{Cramer_1936} admits various equivalent formulations since there is a one to one mapping between variables and their characteristic function: 
\begin{itemize}
\item if the convolution of two proper distributions is a normal distribution, then each of them is a normal distribution.
\item if $\phi_1(t)$ and $\phi_2(t)$ are characteristic functions and if $\phi_1(t)\phi_2(t)=\exp(-\gamma t^2+i\beta t),$ with $\gamma\geq0$ and $-\infty<\beta<\infty,$
then $\phi_j(t)=\exp(-\gamma_j t^2+i\beta_j t)$  with $\gamma_j\geq0$ and $-\infty<\beta_j<\infty.$
\end{itemize}

The two reformulations are obvious as the first one just uses the fact that the distribution of the sum is the convolution of the two distributions. The second uses the one to one mapping between distribution and characteristic functions. It is worth mentioning that the Lévy–Cramér theorem can be generalized to the convolution of two signed measures with restrictions on their negative variation. It has also as a consequence the following implication on the stability of distributions. Closeness of the distribution of a sum of independent random variables to the normal distribution implies closeness of the distribution of each of the summands to the normal distribution, qualitative estimates of the stability are known.
The Lévy–Cramér theorem can also be reformulated in terms of the Poisson distribution and is named the Raikov's theorem. It can also be extended to the convolution of a Poisson and a normal distribution and to other classes of infinitely-divisible distributions \cite{Linnik_1977}. Mathematically, the Levy Cramer theorem writes as follows:\\

\begin{theorem}
If $X_1$ and $X_2$ are two independent variables with $X_1 + X_2$ normally distributed, then both $X_1$ and $X_2$ follow a normal distribution
\end{theorem}

We will provide in the following page a self contained proof that follows the version of \cite{Feller_1971} but provide full details about the assumptions on complex analysis. For this, we will need 4 different lemmas that are presented now.

\subsection{Lemma 1}
%%%%%%%%%%%%%% lemma 1
\begin{lemma}\label{lemma1}
Let $f$ be an entire function without any zero (that does not vanish anywhere), then there exists an entire function $g$ such that $$f =\exp(g)$$
\end{lemma}

\begin{proof}
This result makes sense as intuitively, we would like to use as a candidate $\log( f )$ . However, let us provide a rigorous proof of it and explicitly construct the function $g$. The exponential function is surjective from $\mathbb{C}$  to $\mathbb{C}^*$. This implies that there exists a complex $z_0$ such that  $f(0) =exp(z_0)$. The ratio of two holomorphic functions with the denominator function without any zero being holomorphic and any non vanishing holomorphic function $f$ admitting an entire function as a primitive, the function $\frac{f'}{f}$ defined as the ratio of two non vanishing holomorphic functions admits an entire function as primitive denoted by
$$g(z) =z_0 + \int_{[0,z]}\frac{f'(u)}{f(u)} du$$

Let us define by $v = \exp(g)$. Trivially, we have $ v' = g' \exp(g)  = \frac{f'}{f} v $  
so that $ (v / f )' = \frac{ v' f - v f' } {f^2} = 0$, 
which implies that the ratio $\frac{v}{f}$ is constant and equal to one as by construction, $v(0) = f(0)$. This concludes the proof.
\end{proof}

Let us also prove another elementary result of complex analysis with the following lemma (which is an extension of the Liouville's theorem). Let us denote by  $\Re(z)$ the
real part of the complex number $z$.

\subsection{Lemma 2}
%%%%%%%%%%%%%% lemma 2
\begin{lemma}\label{lemma2}
Let $f$ be an entire function and $P$ a polynomial such that:
$$ \forall z \in \mathbb{C}, \Re(f(z)) \leq |P(z)|,$$
then $f$ is polynomial of degree at most equal to that of $P$.
\end{lemma}
 
\begin{proof}
By definition, the entire function writes as a power series
$$ f(z)=\sum_{n=0}^\infty a_nz^n$$
that converges everywhere in the complex plane (hence uniformly on compact sets). Cauchy formula states that $\forall n \geq 1, \forall r > 0$, we can recover the series coefficients
\begin{equation}\label{eq1} 
a_n r^n = \frac{1}{2 \pi}\int_0^{2 \pi}f(r e^{i \theta}) e ^{-in \theta}d\theta , 
\end{equation}
as well as
\begin{equation}\label{eq2} 
 0 = \frac{1}{2 \pi}\int_0^{2 \pi}f(r e^{i \theta}) e ^{in \theta}d\theta
\end{equation}
and
\begin{equation}\label{eq3} 
a_0 = \frac{1}{2 \pi}\int_0^{2 \pi}f(r e^{i \theta}) d\theta.
\end{equation}

Summing up equation (\ref{eq1}) and the complex conjugate of (\ref{eq2}) gives that 
\begin{equation}\label{eq4}
a_n r^n = \frac{1}{\pi} \int_0^{2 \pi} \Re( f(r e^{i \theta}) )e ^{-in \theta}d\theta , 
\end{equation}

In addition, taking the real part of equation \eqref{eq3} implies that 
\begin{equation}\label{eq3bis}
2 \Re( a_0 ) = \frac{1}{ \pi}\int_0^{2 \pi}\Re( f(r e^{i \theta})) d\theta. 
\end{equation}

which combines with (\ref{eq4}) leads to for $r>0$
\begin{equation}\label{eq5}
| a_n|   r^n + 2 \Re( a_0 ) \leq \frac{1}{\pi} \int_0^{2 \pi} | \Re( f(r e^{i \theta}) ) | + | \Re( f(r e^{i \theta})  | d\theta \leq 4 \,\,  \underset{ |z| = r}{max} (\Re( f(z)), 0 ).
\end{equation}

Let us denote the polynomial function $P=\sum_{k=0}^d b_k X^k$. The hypothesis writes $\Re( f(z)) \leq | \sum_{k=0}^d b_k X^k | \quad \forall z \in \mathbb{C}$, which combined with inequality (\ref{eq5}) leads to:
\begin{equation}\label{eq6}
| a_n|   r^n +  2 \Re( a_0 )  \leq 4 \,\,  \underset{ |z| = r}{max} ( | \sum_{k=0}^d b_k X^k |, 0 ) 
\end{equation}

The dominating term in the RHS of the inequality (\ref{eq6}), for $r$ large, is $|b_d| r^d$. For inequality (\ref{eq6}) to hold, all LHS terms of order $n>d$ should therefore be equal to zero: $\forall n>d, a_n=0$. \end{proof}

Let us now prove a central lemma.
\subsection{Lemma 3}
%%%%%%%%%%%%%% lemma 3
\begin{lemma}\label{lemma3}
If $X$ is a real random variable such that there exists $\eta > 0$ such that $f(\eta) = \int_{\mathbb{R}} e^{ \eta^2 x^2 } d\mathbb{P}_X(x)$ is finite and if the characteristic function has no zero in the complex domain, then $X$ follows a normal distribution
\end{lemma}

\begin{remark}
This is a remarkable property of the normal distribution. It states that the normal is the only distribution whose  characteristic function has no zero such that its transform with the exploding kernel $e^{\eta^2 x^2 }$ is finite. The condition of no zero is necessary as for instance the uniform distribution whose characteristic function is cosinus validates the fact that the function $f(\eta)$ is finite but not the no zero condition on the characteristic function. The characterization of the normal distribution is related to the fact that the normal distribution is the only distribution that is Lévy $\alpha$-stable stable with $\alpha=2$. Another interesting remark is that this property of the normal distribution is closely related to the amazing features of the holomorphic functions as well as the connection between characteristic function and moments.
\end{remark}

\begin{proof}
For $\eta \in \mathbb{R}$ not null, for $\zeta \in \mathbb{C}$ and for $x \in \mathbb{R}$, we have
$| \zeta x| = |\eta x \,\, \eta^{-1} \zeta | \leq \frac{1}{2} (  \eta^2 x^2 + \eta^{-2} |\zeta|^2 ) \leq \eta^2 x^2 + \eta^{-2} |\zeta|^2$ (which is a reformulation of the remarkable identity $2 a b \leq a^2 + b^2$). We can upper bound the integral of $\int_ {\mathbb{R}} | e^{i \zeta x } |  d\mathbb{P}_X(x)$ as follows:
\begin{eqnarray}
\int_ {\mathbb{R}} | e^{i \zeta x } |  d\mathbb{P}_X(x)    \leq   \int_ {\mathbb{R}} e^{| \zeta x |}   d\mathbb{P}_X(x) \leq   \int_ {\mathbb{R}} e^{  \eta^2 x^2 + \eta^{-2} |\zeta|^2} d\mathbb{P}_X(x) 
\end{eqnarray}
The assumption says that for a given $\eta > 0$, $ \int_ {\mathbb{R}}  e^{ \eta^2 x^2}   d\mathbb{P}_X(x)  < \infty$. Hence, for $\zeta \in \mathbb{C}$, the characteristic function of $X$ whose definition in the complex domain is $\phi_X(\zeta) = \int_ {\mathbb{R}} e^{i \zeta x } d\mathbb{P}_X(x)$ is bounded, hence defined since we have 
\begin{eqnarray}\label{ineq1}
| \phi_X(\zeta) | & \leq& \int_ {\mathbb{R}} | e^{i \zeta x } |  d\mathbb{P}_X(x)    \leq  e^{\eta^{-2} |\zeta|^2} \int_ {\mathbb{R}}  e^{ \eta^2 x^2}   d\mathbb{P}_X(x)  < \infty
\end{eqnarray}
In particular, inequality (\ref{ineq1}) provides the dominated convergence necessary to apply the holomorphic theorem under the integral sign to prove that the characteristic function, $\phi_X( \zeta )$ is an entire function. By assumption, it does not have any zero in the complex domain. We can safely apply lemma \ref{lemma1}. So there exists an entire function $g$ such that $\phi_X( \zeta ) = \exp(g( \zeta ))$. Inequality (\ref{ineq1}) states that 
\begin{equation}\label{ineq2} 
\exp(g ( \zeta  )) \leq  e^{\eta^{-2} |\zeta|^2} f(\eta)
\end{equation}
or equivalently 
\begin{equation}\label{ineq2} 
\Re(g(\zeta )) \leq  \eta^{-2} |\zeta|^2 + \ln( f(\eta) )
\end{equation}

Lemma \ref{lemma2} then says that  $g$ is polynomial and at most quadratic. Hence it writes as
\begin{equation}\label{ineq2} 
g(\zeta ) = a_0 + i a_1 \zeta  + \frac{a_2}{2}  \zeta ^2
\end{equation}

In addition, since $\phi_X(0) =1$, we have $a_0 = 0$.  Using the fact that from the characteristic function, we can recover moments with $\mathbb{E}[X^k]=(-i)^k \phi_X^{(k)}(0)$, we get immediately that $a_1=\mathbb{E}[X]=\mu$ and $a_2=-Var(X)=-\sigma^2$, which leads to the fact that the characteristic function of $X$ is of the form
$$
\phi_X( \zeta )  = \exp( 	i \mu \zeta  - \frac{1}{2} \sigma^2 \zeta ^2 )
$$
which is the characteristic function of the normal. This concludes the proof.
\end{proof}

Let us prove a final lemma that reformulates the function $f(\eta)$.

\subsection{Lemma 4}
%%%%%%%%%%%%%% lemma 4
\begin{lemma}\label{lemma4}
The function  $f(\eta)$ can be computed in terms of the probability of the absolute value of $X$, $\mathbb{P}( |X | \geq x)$, as follows:
\begin{equation}
 f(\eta) = 1 +  \int_{0}^{+\infty} 2 x \eta^2 e^{ \eta^2 x^2 } \mathbb{P}( |X | \geq x) dx 
\end{equation}
\end{lemma}

\begin{proof}
 Let us split the integral between negative and positive reals as follows
\begin{eqnarray}
 f(\eta) &= &\int_{-\infty}^{0} e^{ \eta^2 x^2 } d\mathbb{P}_X(x) +  \int_{0}^{+\infty} e^{ \eta^2 x^2 } d\mathbb{P}_X(x)  \label{integral}
\end{eqnarray}

An integration by parts for the first integral, with the two primitives functions chosen judiciously as $e^{ \eta^2 x^2 }$ and $ \mathbb{P}(X \leq x)$, leads to
\begin{eqnarray}
\int_{-\infty}^{0} e^{ \eta^2 x^2 } d\mathbb{P}_X(x) &= & \mathbb{P}(X \leq 0) -  \int_{-\infty}^{0} 2 x \eta^2  e^{ \eta^2 x^2 } \mathbb{P}(X \leq x) dx  \\
& =& \mathbb{P}(X \leq 0) +  \int_{0}^{\infty} 2 x \eta^2  e^{ \eta^2 x^2 } \mathbb{P}(X \leq -x) dx \label{int1}
\end{eqnarray}

Similarly, for the second integral in equation (\ref{integral}), an integration by parts, with the two primitives functions chosen as still $e^{ \eta^2 x^2 }$ but $- \mathbb{P}(X \geq x)$,  leads to
\begin{eqnarray}
\int_{0}^{+\infty} e^{ \eta^2 x^2 } d\mathbb{P}_X(x) & =& \mathbb{P}(X \geq 0) +  \int_{0}^{+\infty} 2 x \eta^2 e^{\eta^2 x^2 } \mathbb{P}(X \geq x) dx \label{int2}
\end{eqnarray}

Combining equations (\ref{int1}) and (\ref{int2}) gives the result.
\end{proof}

We are now able to prove the Levy Cramer theorem that states that if  $X_1$ and $X_2$ are two independent variables with $X_1 + X_2$ normally distributed, then both 
$X_1$ and $X_2$ follow a normal distribution.

%%%%%%%%%%%%%% final proof
\begin{proof}
A first trivial case is when the variance of $X_1+ X_2$ is equal to zero. This means that each of two variables $X_1$ and $X_2$ has a null variance (since they are independent) and the result is trivial as they are all constant variables. We therefore assume that the variance of $X_1+X_2$ is not zero. 
It is worth noticing that since $X_1+X_2$ is a continuous distribution without any atom, we cannot have that both $X_1$ and $X_2$ have an atoms since if it was the case with these two atoms given by $a_1$ and $a2$, we would have
$$
0 < \mathbb{P}(X_1=a_1) \mathbb{P}(X_2=a_2) = \mathbb{P}(X_1=a_1, X_2=a_2) \leq  \mathbb{P}(X_1 + X_2 = a_1+a_2) = 0
$$
which would result in a contradiction! Hence, we can safely assume that $X_2$ has no atom\footnote{if it was not the case, we can interchange $X_1$ and $X_2$ to make this happen} and find $\mu$ that splits the distribution of $X_2$ equally: $\mathbb{P}(X_2 \leq \mu) =  \mathbb{P}(X_2 \geq \mu) = \frac{1}{2}$. Hence, for $ x>0$, we have
\begin{eqnarray}
\mathbb{P}( | X_1 | \geq x ) & =& 2 \mathbb{P}( X_1 \geq x, X_2  \geq \mu  ) +  2 \mathbb{P}( X_1 \leq - x, X_2  \leq \mu )  \\
& \leq & 2 \mathbb{P}( X_1+X_2  \geq x + \mu  ) + 2\mathbb{P}( X_1 + X_2  \leq - x + \mu ) \\
& \leq & 2 \mathbb{P}( | X_1+X_2 - \mu | \geq x  ) \label{ineqp41}
\end{eqnarray}

Lemma \ref{lemma4} states that 
\begin{eqnarray}
 f_{X_1}(\eta) &= & 1 +  \int_{0}^{+\infty} 2 x \eta^2 e^{ \eta^2 x^2 } \mathbb{P}( |X_1 | \geq x) dx 
\end{eqnarray}

Combining these two results gives us a way to upper bound $ f_{X_1}(\eta)$:
\begin{eqnarray}
| f_{X_1}(\eta)|  &\leq &2 + 2 \int_{0}^{+\infty} 2 x \eta^2 e^{ \eta^2 x^2 } \mathbb{P}( | X_1+X_2 - \mu | \geq x) dx \\
 &\leq & 2   f_{X_1+X_2 - \mu}(\eta) 
\end{eqnarray}
The latter inequality is finite as $X_1+X_2 - \mu$ is normal so that $| f_{X_1}(\eta)| \leq \infty$. In addition, since we have $\phi_{X_1} \phi_{X_2} = \phi_{X_1+X_2}$ by independence, the characteristic function of $X_1$ is never equal to zero in the complex domain since $\phi_{X_1+X_2}$ is never null in $\mathbb{C}$. Using lemma \ref{lemma3}, we conclude that $X_1$ is normal. If $X_1$ is of null variance, then $X_2= X_2+X_1 - X_1$ is a normal as the sum of a normal and a constant. Otherwise, $X_1$'s variance is strictly positive and $X_1$ has no atom. Interchanging $X_1$ and $X_2$ and using previous reasoning, we get immediately that $X_2$ follows also a normal distribution, which concludes the proof.
\end{proof}

%%%%%%%%%%% An extension of the Levy Cramer theorem
\section{An extension of the Levy Cramer theorem}
It is a well known and easy to prove result (using for instance characteristic function) that if $X_1$ and $X_2$ are independent normal standard variables, then $X_1 + X_2$ and $X_1 - X_2$ are independent. But in fact, we can prove somehow the opposite that requires less assumption than the Levy Cramer theorem. This is the following result

\begin{theorem}
If $X_1$ and $X_2$ are two independent symmetric variables with mean 0 and variance 1, such that $X_1 + X_2$ and $X_1 - X_2$  are independent, then $X_1$ and $X_2$ follow a normal distribution
\end{theorem}

\begin{remark}
Compared to the Levy Cramer theorem, we do not assume in the hypothesis that any of the variable is normal which makes the result quite remarkable.
\end{remark}

\begin{proof}
Noticing that $2 X_1 = (X_1 + X_2) +( X_1 - X_2)$, we have using characteristic function and independence that 
$$\phi_{2 X_1}(t) = \phi_{X_1+X_2}(t) \phi_{X_1 - X_2}(t)$$
Since $X_1$ and $X_2$  are two symmetric independent variables, we have that \\
$\phi_{X_1+X_2}(t)  = \phi_{X_1}(t)^2 $  and likewise $\phi_{X_1 - X_2}(t)  = \phi_{X_1}(t) \phi_{\bar{X_1}}(t) $
\\\\
We also have that 
$\phi_{2 X_1}(t) = \phi_{X_1}(2 t) $ and 
$ \phi_{\bar{X_1}}(t)  = \overline{\phi_{X_1}}(t)$.
\\
\\
Combining all results, we get
\begin{equation}\label{prop2_eq1}
\phi_{X_1}(2 t) = \phi_{X_1}(t) ^3  \overline{\phi_{X_1}}(t) 
\end{equation}

The convex conjugate of previous equation writes $ \overline{\phi_{ X_1}}(2 t) = \overline{\phi_{X_1}}(t)^3 \phi_{X_1}(t)$, which leads to

\begin{equation}\label{prop2_eq3}
\phi_{ X_1}(2 t)  \overline{\phi_{ X_1}}(2 t) = \phi_{X_1}(t)^4  \overline{\phi_{X_1}}(t)^4
\end{equation}

Denoting by $\psi(t) = \phi_{X_1}(t) \overline{\phi_{X_1}}(t)$, equation (\ref{prop2_eq3}) states that $\psi(2 t) = \psi(t)^4$ or iteratively

\begin{equation}\label{prop2_eq5}
\psi(t) = \psi(\frac{t}{2^n})^{2^n} 
\end{equation}

Clearly, $\phi_{X_1}(t) \neq 0$ for all $t \in\mathbb{R}$, since if it was the case for a given $t_0 \in \mathbb{R}$, then $\psi(t_0)=0$  (as $\psi(t_0) = \phi_{X_1}(t_0)  \overline{\phi_{X_1}}(t_0)$), which would imply that for all $n \in \mathbb{N}$, we would have $\psi(\frac{t}{2^n})=0$, which would imply in particular the following limit 
$\lim\limits_{n  \rightarrow \infty} \psi(\frac{t}{2^n}) = 0$ which will be in contradiction with the continuity of the characteristic function and the fact that $\phi_{X_1}(0)=1$.
Hence, for all $t \in\mathbb{R}$, we have $\phi_{X_1}(t) \neq 0$ and $ \overline{\phi_{X_1}}(t)  \neq 0$. We can define safely
$$\gamma(t) = \frac{ \phi_{X_1}(t) } {\overline{\phi_{X_1}}(t) }.$$
Doing basic computations, we have
\begin{eqnarray}
\gamma(2 t) &= &\frac{ \phi_{X_1}(2 t) } {\overline{\phi_{ X_1}}(2 t) }  =  \frac{ \phi_{X_1}(t) ^3  \overline{\phi_{X_1}}(t)}  {\overline{\phi_{X_1}}(t)^3 \phi_{X_1}(t) }  =  \frac{ \phi_{X_1}(t) ^2 } {\overline{\phi_{X_1}}(t) ^2 }  =  \gamma(t)^2 
\end{eqnarray}
Applying previous equation iteratively, we get
\begin{eqnarray}\label{prop2_eq6}
\gamma( t) = \gamma(\frac{t}{2^n})^{2^n} 
\end{eqnarray}
$\gamma(t)$ 's value in 0 is 1 as $ \phi_{X_1}(0)=1$. As $X_1$ is a variable with zero mean and unit variance, we have that $\phi_{X_1}(t)=1 - \frac{t^2}{2} + o(t^2)$. Hence, $\gamma( t) = 1+ o(t^2)$, which implies in particular that $\forall t \in \mathbb{R}$:
\begin{eqnarray}\label{prop2_eq6}
\gamma( t) = \gamma(\frac{t}{2^n})^{2^n} = \lim\limits_{n  \rightarrow \infty} \gamma(\frac{t}{2^n})^{2^n} =  \lim\limits_{n  \rightarrow \infty} \left( 1+ o(\frac{t}{2^n})\right)^{2^n} = 1
\end{eqnarray}

By definition of $\gamma(t)$, $\phi_{X_1}$ is a real function. Equation (\ref{prop2_eq1}) becomes $\phi_{X_1}(2 t) = \phi_{X_1}(t) ^4$ or iteratively $\phi_{X_1}(t) = \phi_{X_1}(\frac{t}{2^n})^{2^n} $. Remember that $\phi_{X_1}(t)=1 - \frac{t^2}{2} + o(t^2)$ as $X_1$ is a variable with zero mean and unit variance. Hence, we get
\begin{eqnarray}\label{prop2_eq6}
\phi_{X_1}( t) =  \lim\limits_{n  \rightarrow \infty} \phi_{X_1}(\frac{t}{2^n})^{2^n} =  \lim\limits_{n  \rightarrow \infty} \left(1 - \frac{1}{2} \frac{t^2}{2^{n}} + o(t^2) \right)^{2^n} = \exp(- \frac{t^2}{2})
\end{eqnarray}
This shows that $X_1$ 's characteristic function is the one of a normal distribution with zero mean and unit variance, which concludes the proof.
\end{proof}

%%%%% Independence between sample mean and variance and normal distribution %%%%%

\section{Independence between sample mean and variance and normal distribution}
We finally tackle the question of the condition for the sample mean and variance to be independent. This is a strong result that for instance enables us to derive the Student distribution as in the normal case of i.i.d. variables, the sample mean and variance are clearly independent. We are interested in the opposite. What is the condition to impose on our distribution for iid variable to make our sample mean and variance independent? We shall prove that it is only in the case of normal distribution that these two estimators are independent as the following proposal states

\begin{proposition}
The sample mean and variance are independent if and only if the underlying (parent) distribution is normal.
\end{proposition}

\begin{remark} This result was first proved by \cite{Geary_1936} and later by \cite{Lukacs_1942}. We provide a new proof which is simpler as we work with the log characteristic function and the unbiased sample variance. This makes the resulting differential equation trivial to solve as this is just a constant second order derivative constraint. 

This result implies consequently that it will not be easy to derive the underlying distribution of the t-statistic for a non normal distribution. Indeed the t-statistic is defined as the ratio of the sample mean over the sample variance. If the sample mean and sample variance are not independent, the computation of the underlying distribution does not decouple. This makes the problem of the computation of the underlying distribution an integration problem that has no closed form. This kills in particular any hope to derive trivially other distribution that generalizes the case of the Student distribution to non normal underlying assumptions.
\end{remark}

\begin{proof}
The assumption of i.i.d. sample for  $(x_1,\ldots, c_n)$ implies that the joint distribution of $(x_1,\ldots, x_n)$ denoted by
$f_{X_1,\ldots,X_n}(x_1,\ldots, x_n)$ is equal to $\prod_{i=1}^n f_X(x_i)$, which we will write $\prod_{i=1}^n f(x_i)$ dropping the $._{X}$ to make notation lighter.

The log of the characteristic function of the joint variable $(\bar X_n,s_n^2)$ writes
\begin{equation}
\ln( \phi_{(\bar X_n,s_n^2)}(t_1, t_2) ) = \ln \left( \iiint e^{i t_1 \bar x_n + i t_2 s_n^2} \prod_{i=1}^n f(x_i) dx_i \right).
\end{equation}

Similarly, the log of the characteristic function for the sample mean $\bar X_n$  writes 
\begin{equation}
\ln( \phi_{\bar X_n} (t_1) ) =\ln \left(  \iiint e^{i t_1 \bar x_n } \prod_{i=1}^n f(x_i) dx_i \right),
\end{equation}
and similarly for the sample variance
\begin{equation}
\ln( \phi_{s_n^2}(t_2) )= \ln \left(  \iiint e^{ i t_2 s_n^2} \prod_{i=1}^n f(x_i) dx_i \right).
\end{equation}

The assumption of independence between  sample mean $\bar X_n$ and variance $s_n^2$ is equivalent to the fact that the characteristic function of the couple decouples, or that the log characteristic functions sum up.

\begin{equation} \label{condition_1}
\ln( \phi_{(\bar X_n,s_n^2)}(t_1, t_2) ) = \ln( \phi_{\bar X_n} (t_1) ) + \ln( \phi_{s_n^2}(t_2) ).
\end{equation}

Differentiating condition \ref{condition_1} with respect to $t_2$ in $t_2=0$ leads to
\begin{equation} \label{condition_2}
\frac{1}{\phi_{(\bar X_n,s_n^2)}(t_1, t2)} \left. \frac{ \partial \phi_{(\bar X_n,s_n^2)}(t_1, t_2) }{\partial t_2} \right|_{t_2=0} =\frac{1}{\phi_{s_n^2}(t2)}  \left. \frac{ \partial   \phi_{s_n^2}(t_2) }{\partial t_2}  \right|_{t_2=0}.
\end{equation}

Noticing that $\phi_{s_n^2}(0) = 1$ and $\phi_{(\bar X_n,s_n^2)}(t_1, 0) = \phi_{\bar X_n} (t_1)$, the condition \ref{condition_1} writes
\begin{equation} \label{condition_2}
\frac{1}{\phi_{\bar X_n}(t_1)} \left. \frac{ \phi_{(\bar X_n,s_n^2)}(t_1, t_2) }{\partial t_2} \right|_{t_2=0} =\left. \frac{ \partial   \phi_{s_n^2}(t_2) }{\partial t_2}  \right|_{t_2=0}.
\end{equation}

Using the fact that $\bar X_n= \frac{ \sum_{i=1}^n X_i}{n}$, it is easy to see that 

\begin{equation}
 \phi_{\bar X_n} (t_1) = \prod _{i=1}^n \int e^{i t_1 x_i / n} f(x_i) dx_i = [\phi_{X} (t_1 / n)] ^n
\end{equation}

For the sample variance, we can use the "U-statistic" (or symmetric) form  to see that 
\begin{equation}
s_n^2 = \frac{\sum_{i =1}^n X_i^2} {n}  - \frac{\sum_{i \neq j} X_i X_j}{n (n-1)}
\end{equation}

Hence, the derivative of the characteristic function of the couple $(\bar X_n, s_n^2)$ writes 
\begin{align}
\left. \frac{ \partial \phi_{(\bar X_n,s_n^2)}(t_1, t_2) }{\partial t_2} \right|_{t_2=0}  & = \iiint i s_n^2  \prod_{i=1}^n e ^{ i t_1 x_i / n } f(x_i) dx_i   \\
& = i \iiint \left( \frac{\sum_{i =1}^n x_i^2} {n}-  \frac{\sum_{i \neq j} x_i x_j}{n (n-1)}  \right)  \prod_{i=1}^n e ^{ i t_1 x_i / n } f(x_i) dx_i   \\
& = i  [\phi_{X}(\frac{t_1}{n})] ^{n-2} \left( \phi_{X} (\frac{t_1}{n}) \int x^2 e ^{\frac{i  t_1 x}{n}} f(x) dx - (\int x^2 e ^{\frac{i t_1 x}{n}} f(x) dx )^2  \right) 
\end{align}

In the latter equation, if we set $t_1=0$, we get in particular that 
\begin{align}
\left. \frac{ \partial \phi_{s_n^2}(t_1, t_2) }{\partial t_2} \right|_{t_2=0}  & =\left. \frac{ \partial \phi_{(\bar X_n,s_n^2)}(0, t_2) }{\partial t_2} \right|_{t_2=0}  =i \sigma^2
\end{align}

Hence, condition (\ref{condition_2}) writes
\begin{equation} \label{condition_3}
\frac{ \phi_{X} (\frac{t_1}{n}) \int x^2 e ^{\frac{i t_1 x}{n}} f(x) dx - (\int x^2 e ^{\frac{i t_1 x}{n}} f(x) dx )^2 }{[\phi_{X}(\frac{t_1}{n})]^2} = \sigma^2
\end{equation}

We also have that the derivative of the characteristic function $\phi_{X} (\frac{t_1}{n})$ with respect to $u = t_1 / n$ gives
\begin{align}\label{derivative_2}
\frac{ \partial \phi_{X} (u)}{\partial u} & =  \int  i x e ^{i x u} f(x) dx
\end{align}

To simplify notation, we drop the index in $\phi_X$ and writes this function $\phi$. Using equation  (\ref{derivative_2}), condition (\ref{condition_3}) writes 
\begin{align}\label{condition_4}
\frac{ - \phi (u) \frac{ \partial^2 \phi (u)}{\partial u^2} + \left( \frac{ \partial \phi (u)}{\partial u} \right)^2}{\phi (u)^2}  = \sigma^2
\end{align}

The log of the characteristic function of $\phi(u) = \mathbb{E}[ e^{i u X} ]$, denoted by $\Psi(u) =\ln \phi (u)$, first and second derivatives with respect to $u$ are given by:
\begin{align}
\frac{ \partial \Psi(u)}{\partial u} & = \frac{ \partial \ln \phi (u)}{\partial u}  = \frac{1}{ \phi (u)} \frac{ \partial  \phi (u)}{\partial u} &  \quad\\
\frac{ \partial^2 \Psi(u)}{\partial^2 u} & =   \frac{ \partial }{\partial u}  \frac{ \partial \Psi(u)}{\partial u} =  \frac{1}{ \phi (u)} \frac{ \partial^2  \phi (u)}{\partial u^2}  - \frac{1}{ \phi (u)^2 }  \left( \frac{ \partial  \phi (u)}{\partial u} \right)^2 &  \quad
\end{align}

Hence,  condition (\ref{condition_4}) writes 
\begin{align}\label{condition_5}
\frac{ \partial^2 \Psi(u)}{\partial^2 u}   = -\sigma^2
\end{align}

Using the boundary conditions $ \Psi(0)= 0$ and $\Psi^{'}(0)= i  \mathbb{E}[X] = i \mu$, it is easy to integrate condition (\ref{condition_5}) which is a constant second order derivative to get
\begin{align}\label{condition_6}
\Psi(u)  =i  \mu u  - \frac{ \sigma^2 u^2}{2}
\end{align}

Condition (\ref{condition_6}) states that a necessary and sufficient condition for the sample mean and variance to be independent is that the log characteristic function of $X$ is a quadratic form. But a quadratic form for the log characteristic function of $X$  is a characterization of a normal distribution, which concludes the proof.
\end{proof}

\section{Conclusion}
In this paper, we have presented three remarkable properties of the normal distribution: first that if $X_1$ and $X_2$ are two independent variables with $X_1 + X_2$ normally distributed, then both $X_1$ and $X_2$ follow a normal distribution, which is referred to as the Cramer theorem, second a variation of the Levy Cramer theorem that states that if $X_1$ and $X_2$ are independent symmetric random variables with finite variance with $X_1+X_2$ and $X_1 - X_2$ independent, then both $X_1$ and $X_2$ are normal, and third that the normal distribution is characterized by the fact that it is the only distribution for which the sample mean and variance are independent (which is a central property for deriving the Student distribution and referred as the Geary theorem (see \cite{ Geary_1936})). 
\\\\
The novelty of this paper has been to provide new, quicker or self contained proofs of theses theorems. In particular, we revisited the proof first provided by \cite{Feller_1971} to give full details of it as the sketch of the proof provided was rather elusive on technical details in complex analysis. Also using the log characteristic function turns out to provide a quicker proof of the result on the characterization of the normal distribution as the only distribution for which sample mean and variance are independent.

\bibliography{mybib}

\end{document}